\documentclass[12pt]{amsart}

\usepackage{amsmath,amssymb,amsfonts,amscd,pictexwd,dcpic}
\usepackage{booktabs}
\usepackage{dcolumn}
\usepackage{amsthm}
\usepackage{amsmath}
\usepackage{amssymb}
\usepackage{scrextend}
\usepackage{comment}
\usepackage{float}
\usepackage{tikz}
\usepackage{subcaption}
\usepackage{tkz-graph}

\newtheorem{theorem}{Theorem}[section]
\newtheorem{Theorem}[theorem]{Theorem}
\newtheorem{Lemma}[theorem]{Lemma}
\newtheorem{proposition}[theorem]{Proposition}

\newtheorem{Corollary}[theorem]{Corollary}

\makeatletter
\newcommand{\thickhline}{%
    \noalign {\ifnum 0=`}\fi \hrule height 1pt
    \futurelet \reserved@a \@xhline
}
\makeatother
\title{Power Graphs of Finite Group}

\author{Amrita Acharyya}
\address{Department of Mathematics and Statistics\\
University of Toledo, Main Campus\\
Toledo, OH 43606-3390}
\email{Amrita.Acharyya@utoledo.edu}

\author{Allen Williams }
\address{Department of Mathematics and Statistics\\
University of Toledo, Main Campus\\
Toledo, OH 43606-3390}
\email{Allen.Williams2@rockets.utoledo.edu}



\begin{document}
\begin{abstract}
The Directed Power Graph of a group is a graph whose vertex set is the elements of the group, with an edge from $x$ to $y$ if $y$ is a power of $x$.  The \textit{Power Graph} of a group can be obtained from the directed power graph by disorienting its edges. This article discusses properties of cliques, cycles, paths, and coloring in power graphs of finite groups.  A construction of the longest directed path in power graphs of cyclic groups is given, along with some results on distance in power graphs.  We discuss the cyclic subgroup graph of a group and show that it shares a remarkable number of properties with the power graph, including independence number, completeness, number of holes etc., with a few exceptions like planarity and Hamiltonian.
\end{abstract}


\maketitle
\section{Introduction}
The power graph of finite groups has been studied in \cite{Power finite}, \cite{power semi}, \cite{power combi}, \cite{Power finiteII}.
In this work, we denote the group of integers mod $n$ under addition by $\mathbb{Z}_n$, or $\mathbb{Z} / n\mathbb{Z}$.  These notations are used interchangeably in this article.  It can be shown that any cyclic group of order $n$ is isomorphic to $\mathbb{Z}_n$. 

A $\textit{graph}$ is a pair $\Gamma=(V(\Gamma),E(\Gamma))$ where $V(\Gamma)$ is a non-empty set, called the \textit{vertex set} whose elements are called \textit{vertices}, and $E(\Gamma)$ is a (possibly empty) set consisting of sets of pairs of elements of $V(\Gamma)$ called the $\textit{edge set}$, whose elements are called $\textit{edges}.$  If $e\in E(\Gamma)$ is an edge and $v\in e$, then $e$ is said to be $\textit{incident }$to $v$.  The \textit{degree} of a vertex $v$ is the number of edges incident to $v$.  If $\{v_1,v_2\}\in E(\Gamma)$, then $v_1$ and $v_2$ are said to be adjacent. Where there is no ambiguity, $V(\Gamma)$ is sometimes denoted as just $V$, and similarly $E(\Gamma)$ as E.  A $\textit{directed graph}$ or digraph is a pair $\Gamma=(V(\Gamma),E(\Gamma))$ where $V(\Gamma)$ is a non-empty vertex set, and $E$ is a (possible empty) set consisting of ordered pairs of elements in $V(\Gamma)$ called the edge set.  

Let $G$ be a group.  The \textit{power graph} of $G$, denoted by $\mathfrak{g}(G)$ is defined as the graph with vertex set consisting of the elements of $G$, and edge set $E=\{\{x,y\}\mid x\neq y $ and $\langle x\rangle \leq \langle y\rangle$ or $\langle y\rangle \leq \langle x\rangle \}$. The \textit{directed power graph} of $G$, denoted $\vec{\mathfrak{g}}(G)$ is defined as the graph with vertex set consisting of the elements of $G$, and edge set $E=\{(x,y)\mid x\neq y $ and $\langle y\rangle \leq \langle x\rangle$\}.

\section{Some Properties of Power Graphs of Groups}
A graph $\Gamma$ is called $\textit{connected}$ if for any two vertices $u$ and $v$ in $\Gamma$, there is a path joining $u$ and $v$.  In a connected graph $\Gamma$, the distance between $u$ and $v$, denoted $d_{\Gamma}(u,v)$, can be defined as the length of the shortest path joining $u$ and $v$.  The $\textit{eccentricity}$ of a vertex $u$ is the maximum distance between $u$ and any other vertex in $\Gamma$.  The $\textit{radius}$ of $\Gamma$ is the minimum eccentricity of a vertex in $\Gamma$, and the $\textit{diameter}$ of $\Gamma$ is the maximum eccentricity of a vertex in $\Gamma$.  A central vertex in $\Gamma$ is a vertex with eccentricity equal to the radius of $\Gamma$.  The $\textit{center}$ of $\Gamma$ is the set of central vertices in $\Gamma$.

\begin{proposition}
It is well known that the Power graphs of finite groups are connected.
\end{proposition}
\begin{proof}
Let $G$ be a group with power graph $\mathfrak{g}(G)$ and identity $e$.  Let $g$ be an arbitrary element of $G$.  Since $G$ is a group, $\langle e\rangle \leq \langle g\rangle$, so in the power graph of a group there is an edge between the identity and every other group element.
\end{proof}

Note that the distance between the identity and any other vertex is $1$, in the power graph of a group, giving the following corollary.

\begin{Corollary}
The radius of the power graph of a group is $1$, and the center of the power graph of a group of order $n$ is the set of vertices with degree $n-1$.
\end{Corollary}

The $\textit{center}$ of a group $G$ is the set of elements of $G$ which commute with all other elements of $G$.  The following proposition gives a relation between the center of a power graph of a group and the center of the group.

\begin{proposition}
Let $G$ be a group of order $n$ with power graph $\mathfrak{g}(G)$. The vertices in the center of the power graph $\mathfrak{g}(G)$ are in the center of the group $G$.
\end{proposition}
\begin{proof}
Let $g\in G$ and suppose $deg_{\mathfrak{g}(G)}(g)=n-1$.  Then for any other $h\in G$, either $\langle h\rangle \leq \langle g\rangle$, or $\langle g\rangle \leq \langle h\rangle$, either way $g$ and $h$ commute.  Hence, $g$ is in the center of $G$.
\end{proof}

Since power graphs of groups are connected and the identity is adjacent to every other element, there is a path of length at most $2$ between any two vertices of a power graph of a group, so the diameter of the power graph of a group is $2$, unless the power graph is complete, in which case it is $1$.

\begin{proposition}
Let $G$ be a non trivial Abelian group of order $n$.  The center of $\mathfrak{g}(G)$ has cardinality
\begin{enumerate}
    \item $\phi(n)+1$, if $G$ is cyclic and $n\neq p^k$ for any prime $p$ and positive integer $k$.
    \item $n$, if $G$ is cyclic and $n=p^k$ for some prime $p$ and positive integer $k$.
    \item $1$, if $G$ is non-cyclic.
\end{enumerate}
\end{proposition}
\begin{proof}$\phantom e$
\begin{enumerate}
\item[(1):] First let $G$ be a cyclic group whose order is $n$, and suppose $n$ is not a power of a prime.  Note, all the $\phi(n)$ many generators of $G$ and the identity are in the center of $\mathfrak{g}(G)$, so the center has size at least $\phi(n)+1$.  Since $n$ is not a power of a prime, $n$ has at least two distinct prime factors. Consider any $b\in G$, where $b$ is neither a generator of $G$, nor the identity of $G$. So, if order of $b$ is $m$, then $m<n$. Hence, by converse of Lagrange's Theorem over finite Abelian group, there exists $c\in G$, where order of $c$ is $r$ and $r$ does not divide $m$, $m$ does not divide $r$. So, there is no edge between $b$ and $c$. Since $b$ is an arbitrary non-generator of $G$, which is not the identity of $G$ either, the cardinality result  follows.

\item[(2):] Now suppose $G$ is cyclic and it's order is $n=p^k$ for some prime $p$ and positive integer $k$. Then, the graph being complete, by \cite{power semi}, the center has cardinality $n$.
    
\item[(3):] Let, $G$ be not cyclic. Then, the result follows by the proposition \ref{prop 8} in the next section.
\end{enumerate}
\end{proof}
\section{Some results relating Composition series in power graphs}

     The Jordan-H\"{o}lder theorem asserts that any two composition series of a given group have the same length, and isomorphic factor groups up to permutation.  The length of a composition series of a group $G$ is called its \textit{composition length} and denoted as $\ell$(G). The composition length $\ell(G)$ is also the maximum length a normal series of $G$ can have.  Then a group $G$ can be partitioned in the following way, let $X_i$ be the set of all elements of $G$ which generate a cyclic subgroup with composition length $i$. Then $G=\bigcup_i X_i$, and $X_i\cap X_j=\varnothing$ for any $i\neq j$.
    
    \begin{Lemma}
    The composition length of an Abelian group is the sum of the exponents of the order of the group, when written as a product of prime numbers.
    \end{Lemma}
    \begin{proof}
    An Abelian group is simple only if its order is prime.  Then if $\lvert G\rvert=p_1^{k_1}p_2^{k_2}\cdots p_n^{k_n}$, then each $G_{i}/G_{i+1}$ must have prime order, that is $\lvert G_{i}\rvert=p_i\lvert G_{i+1}\rvert$.  Then each $G_{i-1}$ has an order whose exponents sum to one less than those in the order of $G_i$, so there must be $k_1+k_2+\cdots+k_n$ inclusions in a maximal length normal series of an Abelian group.
    \end{proof}
    \begin{proposition}
    If $X_i$ is non-empty, then $X_{i-1}$ is non-empty.
    \end{proposition}
    \begin{proof}
    Let $x\in X_i$ and $\langle x\rangle\supset \langle x_{i-1}\rangle\supset\cdots\langle x_2\rangle \supset \{e\}$ be a composition series.  Then there is a normal series of length $i-1$ from $\langle x_{i-1}\rangle$. Suppose there were a normal series of length $i$ or greater from $\langle x_{i-1}\rangle$, then $\langle x\rangle \supset \langle x_{i-1}\rangle \supset\cdots\supset \{e\}$ is a normal series from $\langle x\rangle$ of length greater than $i$, contradicting the assumption that $x\in X_i$.  
    \end{proof}
    \begin{Corollary}
    Removing subgroups from the beginning of a composition series of length greater than zero from a cyclic subgroup leaves a composition series.
    \end{Corollary}
    \begin{proposition}
    If some vertex in $X_i$ is adjacent to all other vertices in $X_i$, then $X_i$ is a clique in $\mathfrak{g}(G)$.  
    \end{proposition}
    \begin{proof}
     
    If $x_1,x_2\in X_i$ and they are adjacent in $\mathfrak{g}(G)$, then either $\langle x_1\rangle \leq \langle x_2\rangle$ or $\langle x_2\rangle \leq \langle x_1\rangle$.  Without loss of generality suppose $\langle x_2\rangle$ is a proper subgroup of $\langle x_1\rangle$.  Then there exists a composition series $\langle x_2\rangle\supset \langle x_3\rangle\supset\cdots\supset\{e\}$ of length $i$ and $\langle x_1\rangle\supset \langle x_2\rangle\supset\cdots\supset\{e\}$ is a normal series of length $i+1$, contradicting the assumption that $x_1\in X_i$, so it must be the case that $\langle x_2\rangle=\langle x_1\rangle$.  Then, if any element of $X_i$ is adjacent to all other elements of $X_i$ in $\mathfrak{g}(G)$, then all elements of $X_i$ generate the same subgroup and form a clique in $\mathfrak{g}(G).$
    \end{proof}
    \begin{proposition}
    If $X_i$ is a clique in $\mathfrak{g}(G)$, then $X_{i+1}$ is a clique in $\mathfrak{g}(G)$.
    \end{proposition}
    \begin{proof}
    Suppose $X_i$ is a clique in $\mathfrak{g}(G)$ and $X_{i+1}$ is not a clique in $\mathfrak{g}(G)$, and let $x\in X_i$.  Then there exist two distinct elements in $X_{i+1}$, say $y$ and $z$, with orders $p\lvert x\rvert$ and $q\lvert x\rvert$ respectively where $p$ and $q$ are prime and $p\neq q$.  Write $\lvert x\rvert=p_1^{k_1}p_2^{k_2}\cdots p_n^{k_n}$, with $\sum_{t=1}^n k_t=i$.  Then $\langle y\rangle$ has a subgroup of order $pp_1^{k_1-1}p_2^{k_2}\cdots p_n^{k_n}$, and $\langle z\rangle$ has a subgroup of order $qp_1^{k_1-1}p_2^{k_2}\cdots p_n^{k_n}$.  Both of these subgroups are in $X_i$, but since $p\neq q$ at least one of the subgroups is not equal to $\langle x\rangle$, contradicting the assumption that $X_i$ is a clique in $\mathfrak{g}(G)$.
    \end{proof}
    \begin{proposition}
    If $X_i$ is a clique in $\mathfrak{g}(G)$ and $x\in X_i$ and $\langle x\rangle \neq p^k$ for some prime $p$ and positive integer $k$, then $X_{i+1}$ is empty.
    \end{proposition}
    \begin{proof}
    Suppose $x\in X_i$, $\langle x\rangle \neq p^k$, so $\lvert x\rvert$ has at least two distinct prime factors, say $p$ and $q$.  Suppose there exists a $y\in X_{i+1}$, then $pq \mid \lvert y\rvert$, but then $\lvert y\rvert=pqm$ for some $m$ (possible equal to $p$ or $q$), and in that case $\langle y\rangle$ has subgroups of orders $pm$ and $qm$, both in $X_i$, contradicting the assumption that $X_i$ is a clique in $\mathfrak{g}(G)$.
    \end{proof}
    \begin{proposition}\label{prop 8}    If $G$ is a finite Abelian group and not cyclic, then $max_iX_i$ is not a clique in $\mathfrak{g}(G)$. 
    \end{proposition}
    \begin{proof}
    Since $G$ is a finite non-cyclic Abelian group, $G\cong \mathbb{Z}_{p_1}^{k_1}\times\mathbb{Z}_{p_2}^{k_2}\times\cdots\times\mathbb{Z}_{p_n}^{k_n}$, where not all of the $p_{i}'s$ are distinct.  Then, $G\cong \mathbb{Z}_m\times\mathbb{Z}_{p_i}^{k_i}\times\cdots\times\mathbb{Z}_{p_t}^{k_t}$, where $p_i^{k_i}$ $\mid$ $m,\cdots$ $p_t^{k_t}$ $\mid$ $m$, and $m=q_1^{l_1}q_2^{l_2}\cdots q_r^{l_r}$ where $q_1,q_2,\cdots,q_r$ are relatively prime.  Notice that $g=(1,0,\cdots,0)$ and $g'=(1,1,0,\cdots,0)$ both have the maximum possible order in $G$, that is $m$. Also, notice that $\langle g\rangle \neq \langle g'\rangle$, so $g$ and $g'$ are both in $max_iX_i$ but not adjacent to each other, so $max_iX_i$ is not a clique in $\mathfrak{g}(G)$.
    \end{proof}
    Since $max_iX_i$ is not a clique in $\mathfrak{g}(G)$, then no $X_i$ is a clique in $\mathfrak{g}(G)$.  Since $G=\bigcup_iX_i$, no non-identity vertex in $\mathfrak{g}(G)$ has degree $n-1$.
\section{Perfect Graphs}
A \textit{clique} in a graph $\Gamma$ is a subset $C$ of the vertices of $\Gamma$, such that any two vertices in $C$ are adjacent. The size of the largest clique in $\Gamma$ is called the \textit{clique number} of $\Gamma$, and is denoted $\omega(\Gamma)$.

A \textit{vertex coloring} of a graph is an assignment of labels (called colors) to the vertices of a graph $\Gamma$ such that no two adjacent vertices share the same color.  The smallest number of colors which can be used in a coloring of $\Gamma$ is called the \textit{chromatic number} of $\Gamma$, and is denoted $\chi(\Gamma)$. 

A graph $\Gamma$ is called \textit{perfect} if for every induced subgraph $\Gamma_i$ of $\Gamma$, $\omega(\Gamma_i)=\chi(\Gamma_i).$ It was conjectured by Berge in 1961, and was proved by Chudnovsky et. al. in \cite{PerfectGraphs} that graphs are perfect if and only if they contain no holes or anti-holes of odd length greater than $3$. A \textit{hole} in a graph is a cycle such that no two vertices in the cycle are joined by an edge which does not itself belong to the cycle.  An anti-hole is the edge complement of a hole.

\begin{Theorem}[Strong Perfect Graph Theorem]
    A graph is perfect if and only if it contains no odd holes or odd anti-holes of length greater than 3.
\end{Theorem}

Let $G$ be a group with power graph $\mathfrak{g}(G)$.  Then $\mathfrak{g}(G)$ can contain no holes of odd length.  To prove this a few short lemmas are used.
\begin{Lemma}[Path]
Let $G$ be a group with directed power graph $\vec{\mathfrak{g}}(G)$.  If a path exists between two vertices then they are adjacent.
    \begin{proof}
        Denote the start of the path as vertex $a$ and label the vertices along the path as $a_1,a_2,\cdots a_m$.  Then $a_1=a^n,a_2={a_1}^{n_1},a_3={a_2}^{n_2}\cdots a_m={a_{m-1}}^{n_{m-1}}$.  Then $a_m=a^{nn_1n_2\cdots n_{m-1}}$, so there is an edge from $a$ to $a_m$.
    \end{proof}
\end{Lemma}
\begin{Lemma}[Strong Path]
Let $G$ be a group with directed power graph $\vec{\mathfrak{g}}(G)$, and let $\vec{\mathfrak{g}}(G)$ contain a directed path of length $n$, then the vertices making up the directed path form a clique of size $n$ in $\mathfrak{g}(G)$.
\end{Lemma}
\begin{proof}
 The proof follows from the previous (Path) Lemma.
\end{proof}

The strong path lemma shows that whenever there is a path of length $n$ in $\vec{\mathfrak{g}}(G)$  there is a clique of size $n$ consisting of the same vertices in $\mathfrak{g}(G)$.  Here it will be shown that the converse is true as well, that is, whenever there is a clique of size $n$ in $\mathfrak{g}(G)$, then those vertices are traversable by a path in $\vec{\mathfrak{g}}(G)$.  Path-clique equivalence in power graphs of finite groups is a consequence of R\'{e}dei's theorem~\cite{Redei}.

As per our need, later in our work, we reproduced the proof of the following well known theorem here one more time.

\begin{Theorem}[R\'{e}dei's Theorem]
    Every orientation of a complete graph contains a directed Hamiltonian path.
\end{Theorem}

\begin{Theorem}[Path-clique Equivalence]\label{theo 3}
Let $G$ be a group with directed power graph $\vec{\mathfrak{g}}(G)$ and undirected power graph $\mathfrak{g}(G)$.  Then whenever there is a path in $\vec{\mathfrak{g}}(G)$, its constituent vertices form a clique in $\mathfrak{g}(G)$, and whenever there is a clique in $\mathfrak{g}(G)$, its vertices are traversable by a path in $\vec{\mathfrak{g}}(G)$. 
\end{Theorem}
\begin{proof}
The proof of the first direction is given above, here we show that a clique in $\mathfrak{g}(G)$ is traversable by a path in $\vec{\mathfrak{g}}(G)$.  Let $\mathfrak{g}(G)$ contain a clique of size $\alpha$.  The proof is by induction on $\alpha$.  If $\alpha=1$, then the clique is traversable by the path consisting of only the single vertex in the clique.

Suppose the result holds for $0<\alpha \leq k$, and let there exist in $\mathfrak{g}(G)$ a clique of size $k+1$.  By the induction hypothesis, a clique of size $k$ is traversable by a directed path, so at most one vertex is excluded from the longest path through the clique. Proceed along that directed path through these vertices and label them in the order they are encountered, $v_1, v_2, v_3,\cdots,v_k$. If for any $v_i$ in the clique both $(v_i,v_{k+1})$ and $(v_{k+1},v_i)$ are edges in $\vec{\mathfrak{g}}(G)$, then the sequence can be modified from $(\cdots v_i,v_{i+1}\cdots)$ to $(\cdots v_i, v_{k+1}, v_{i+1}\cdots)$ adding $v_{k+1}$ to the path.  Suppose no two-sided edges exist in the clique. If $(v_{k+1}, v_1)$ is an edge in $\vec{\mathfrak{g}}(G)$, then $v_{k+1}$ can be added to the beginning of the existing path.  If $(v_k,v_{k+1}) $ is an edge in $\vec{\mathfrak{g}}(G)$, then $v_{k+1}$ can be added to the end of the existing path. Suppose neither of these are edges in $\vec{\mathfrak{g}}(G)$.  If $(v_{k+1}, v_2)$ is an edge is $\vec{\mathfrak{g}}(G)$, then $v_{k+1}$ can be inserted in the path between $v_1$ and $v_2$.  Then, if $(v_{k+1}, v_2)$ is not an edge in $\vec{\mathfrak{g}}(G)$, $(v_2,v_{k+1})$ must be an edge in $\vec{\mathfrak{g}}(G)$.  Proceeding in this way we get the desired directed path in $\vec{\mathfrak{g}}(G)$.
\end{proof}

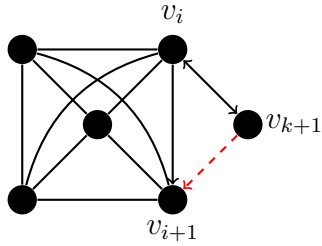
\begin{figure}[h]
\centering
\caption{An illustration of Theorem \ref{theo 3}.  Only relevant edges from the new vertex are shown in each case.}
\begin{subfigure}[b]{0.45\textwidth}
\centering

    \begin{tikzpicture}
            \node[circle,fill] (A) at (0,0) {};
            \node[circle,fill] (B) at (-1,-1) {};
            \node[circle,fill] (C) at (1,-1) [label={[yshift=-0.9cm]$v_{i+1}$}] {};
            \node[circle,fill] (D) at (-1,1) {};
            \node[circle,fill] (E) at (1,1) [label={$v_i$}] {};
            \node[circle,fill] (F) at (2,0) [label={[yshift=-0.48cm, xshift=0.62cm]$v_{k+1}$}] {};
            \Edge(A)(B);
            \Edge(A)(C);
            \Edge(A)(D);
            \Edge(A)(E);
            \Edge(B)(C);
            \Edge(B)(D);
            \Edge[style={bend left}](B)(E);
            \Edge[style={bend right}](C)(D);
            \Edge[style={<-}](C)(E);
            \Edge(D)(E);
            \Edge[style={<->}](E)(F)
            \Edge[style={->, dashed}, color=red](F)(C)
    \end{tikzpicture}
    \caption{An edge in both directions between $v_{k+1}$ and any other vertex implies $v_{k+1}$ can be added to the path}
\end{subfigure}
\qquad
\begin{subfigure}[b]{0.45\textwidth}
\centering

    \begin{tikzpicture}
            \node[circle,fill] (A) at (0,0) [label={[yshift=-0.9cm]$v_{3}$}] {};
            \node[circle,fill] (B) at (-1,-1) [label={[yshift=-0.9cm]$v_{1}$}] {};
            \node[circle,fill] (C) at (1,-1) [label={[yshift=-0.9cm]$v_{5}$}] {};
            \node[circle,fill] (D) at (-1,1) [label={$v_2$}] {};
            \node[circle,fill] (E) at (1,1) [label={$v_4$}] {};
            \node[circle,fill] (F) at (2,0) [label={[yshift=-0.48cm, xshift=0.62cm]$v_{6}$}] {};
            \Edge[style={->}](A)(B);
            \Edge[style={->}](A)(C);
            \Edge[style={<-}](A)(D);
            \Edge[style={->}](A)(E);
            \Edge[style={->}](B)(C);
            \Edge[style={->}](B)(D);
            \Edge[style={->,bend left}](B)(E);
            \Edge[style={->,bend right}](C)(D);
            \Edge[style={<-}](C)(E);
            \Edge[style={<-}](D)(E);
            \Edge[style={<-, dashed}, color=red](F)(C)
    \end{tikzpicture}
    \caption{If there is an edge from $v_k$ to $v_{k+1}$, then $v_{k+1}$ can be added to the end of the path.  Here $k=5$.}
\end{subfigure}
\linebreak
\begin{subfigure}[b]{0.45\textwidth}
\centering

    \begin{tikzpicture}
            \node[circle,fill] (A) at (0,0) [label={[yshift=-0.9cm]$v_{3}$}] {};
            \node[circle,fill] (B) at (-1,-1) [label={[yshift=-0.9cm]$v_{1}$}] {};
            \node[circle,fill] (C) at (1,-1) [label={[yshift=-0.9cm]$v_{5}$}] {};
            \node[circle,fill] (D) at (-1,1) [label={$v_2$}] {};
            \node[circle,fill] (E) at (1,1) [label={$v_4$}] {};
            \node[circle,fill] (F) at (2,0) [label={[yshift=-0.48cm, xshift=0.62cm]$v_{6}$}] {};
            \Edge[style={->}](A)(B);
            \Edge[style={->}](A)(C);
            \Edge[style={<-}](A)(D);
            \Edge[style={->}](A)(E);
            \Edge[style={->}](B)(C);
            \Edge[style={->}](B)(D);
            \Edge[style={->,bend left}](B)(E);
            \Edge[style={->,bend right}](C)(D);
            \Edge[style={<-}](C)(E);
            \Edge[style={<-}](D)(E);
            \draw[<-, dashed, color=red] (B) to[out=-30,in=-90,swap] (F);
    \end{tikzpicture}
    \caption{If there is an edge from $v_{k+1}$ to $v_{1}$, then $v_{k+1}$ can be added to beginning of the the path.  Here $k=5$.}
\end{subfigure}
\qquad
\begin{subfigure}[b]{0.45\textwidth}
\centering

    \begin{tikzpicture}
            \node[circle,fill] (A) at (0,0) [label={[yshift=-0.9cm]$v_{i}$}] {};
            \node[circle,fill] (B) at (-1,-1) {};
            \node[circle,fill] (C) at (1,-1) [label={[yshift=-0.9cm]$v_{i+1}$}] {};
            \node[circle,fill] (D) at (-1,1) {};
            \node[circle,fill] (E) at (1,1) {};
            \node[circle,fill] (F) at (2,0) [label={[yshift=-0.48cm, xshift=0.62cm]$v_{k+1}$}] {};
            \Edge[style={->}](A)(B);
            \Edge[style={->}](A)(C);
            \Edge[style={<-}](A)(D);
            \Edge[style={->}](A)(E);
            \Edge[style={->}](B)(C);
            \Edge[style={->}](B)(D);
            \Edge[style={->,bend left}](B)(E);
            \Edge[style={->,bend right}](C)(D);
            \Edge[style={<-}](C)(E);
            \Edge[style={<-}](D)(E);
            \Edge[style={->, dashed}, color=red](A)(F);
            \Edge[style={->, dashed}, color=red](F)(C);
    \end{tikzpicture}
    \caption{If none of the other cases hold, then there must exist a pair of vertices in between which $v_{k+1}$ can be placed.}
\end{subfigure}
\end{figure}

\setcounter{figure}{2}

It is well known that 
Power graphs of groups are perfect, which has been proved by using
 Theorem~\cite{PerfectGraphs}.

\section{The Cyclic Subgroup Graph}

Let $G$ be a group and define the relation $\sim$ on $G$ by $x\sim y$ if $\langle x\rangle=\langle y\rangle.$  Define the graph $\vec{C}(G)$ by $V(\vec{C}(G))= G / \mathord{\sim}$ and $(A,B)\in E(\vec{C}(G))$ if there exists elements $b\in B$ and $a\in A$ such that $\langle b\rangle \leq \langle a\rangle$ and $\langle b \rangle \neq \langle a\rangle.$  Also define a weight function $w:V(\vec{C}(G))\to \mathbb{N}$ by $w(A)=\lvert A\rvert$.  Then $\vec{C}(G)$ is a directed acyclic graph with a similar structure to the directed power graph $\vec{\mathfrak{g}}(G).$ \begin{proposition}The weight of the path with the largest weight in $\vec{C}(G)$ is the length of the longest path in $\vec{\mathfrak{g}}(G).$
\end{proposition}
\begin{proof}
The vertices in $\vec{C}(G)$ with only out-edges represent generators of the maximal cyclic subgroups of $G$.  As in the proof of Theorem 3 above any longest path in $G$ must start with these vertices.  If this maximal cyclic subgroup has order $n$ then a vertex adjacent to it represents generators of a cyclic subgroup of order $\frac{n}{d}$ where $d$ is a divisor of $n$, and if the subgroup is maximal $d$ will be a prime divisor of $n$. Then the sum of the weights of the vertices in a path from a maximal cyclic subgroup of $G$ of order $n$ through all of its maximal subgroups of maximum order to the trivial subgroup will be given by $\Psi(n)$, the length of the longest path in $\vec{\mathfrak{g}}(G)$.  
\end{proof}
When $G$ is cyclic there is no ambiguity in naming vertices in $\vec{C}(G)$ by their corresponding isomorphic group $\mathbb{Z}_n$, for example
\begin{figure}[H]
    \caption{$\vec{C}(\mathbb{Z}_{18})$ with vertex weights shown in parentheses.}
\begin{tikzpicture}
        \node[shape=circle, draw=black] (A) at (0,0) {$\mathbb{Z}_{6}(2)$};
        \node[shape=circle, draw=black] (B) at (1.5,1) {$\mathbb{Z}_{18}(6)$};
        \node[shape=circle, draw=black] (C) at (3,0) {$\mathbb{Z}_{9}(6)$};
        \node[shape=circle, draw=black] (D) at (3,-2) {$\mathbb{Z}_{3}(2)$};
        \node[shape=circle, draw=black] (E) at (1.5,-3) {$\{e\}(1)$};
        \node[shape=circle, draw=black] (F) at (0,-2) {$\mathbb{Z}_{2}(1)$};
        \path [->] (B) edge node[left] {} (A);
        \path [->] (B) edge node[left] {} (C);
        \path [->] (B) edge node[left] {} (D);
        \path [->] (B) edge node[left] {} (E);
        \path [->] (B) edge node[left] {} (F);
        \path [->] (A) edge node[left] {} (D);
        \path [->] (A) edge node[left] {} (E);
        \path [->] (A) edge node[left] {} (F);
        \path [->] (F) edge node[left] {} (E);
        \path [->] (D) edge node[left] {} (E);
        \path [->] (C) edge node[left] {} (D);
        \path [->] (C) edge node[left] {} (E);
        
    \end{tikzpicture}
\end{figure}
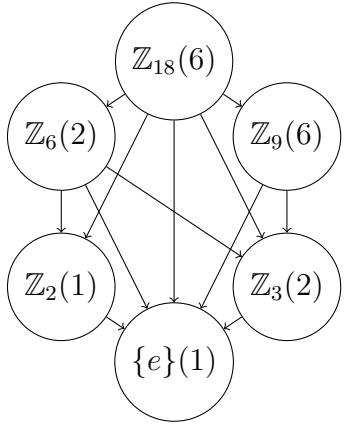
The non-oriented graph $C(G)$ also contains a lot of the information in the power graph in a smaller form.  For example $d_{\mathfrak{g}(G)}(u,v)=d_{C(G)}([u]_{\sim},[v]_{\sim})$ as long as $[u]_{\sim}\neq [v]_{\sim}$ and the independence number $\alpha(\mathfrak{g}(G))=\alpha(C(G))$.

\begin{proposition}
The cyclic subgroup graph is isomorphic to an induced subgraph of the power graph.
\end{proposition}
\begin{proof}
Define $f:V(C(G))\to(V\mathfrak{g}(G))$ by $f([x]_{\sim})=x^{'}$, for a fixed choice of the corresponding equivalence class representative $x^{'}$. The result follows by the definition of cyclic subgroup graph.
\end{proof}

\begin{proposition}
$d_{\mathfrak{g}(G)}(u,v)=d_{C(G)}([u]_{\sim},[v]_{\sim})$ as long as $[u]_{\sim}\neq [v]_{\sim}$.
\end{proposition}
\begin{proof}
let $G$ be a group and let $u,v\in G$ with $d_{\mathfrak{g}(G)}(u,v)=k$.  
Let $P=\{u, u_1,u_2,\cdots,u_k, v\}$ denote the shortest path from $u$ to $v$ in $\mathfrak{g}(G)$.  It must be the case that for each $u_i,u_j\in P$, $[u_i]_{\sim}\neq[u_j]_{\sim}$ otherwise deleting the $u_i$ or $u_j$ we get a shorter path in $\mathfrak{g}(G)$ which is a contradiction.  Then $\{[u]_{\sim},[u_1]_{\sim},\cdots,[u_k]_{\sim}\}$ is a path from $[u]_{\sim}$ to $[v]_{\sim}$ in $C(G)$.  Suppose there is a shorter path from $[u]_{\sim}$ to $[v]_{\sim}$ in $C(G)$, and denote that path $P_1=\{[u]_{\sim},[y_1]_{\sim},\cdots[y_j]_{\sim},[v]_{\sim}\}$ where $j<k$.  Then $\{u,y_1,\cdots, y_j,v\}$ is a path in $\mathfrak{g}(G)$ which is shorter than $P$, a contradiction.  Then $\{[u]_{\sim},[u_1]_{\sim},\cdots,[u_k]_{\sim}\}$ is the shortest path from $[u]_{\sim}$ to $[v]_{\sim}$ in $C(G)$ meaning $d_{C(G)}([u]_{\sim},[v]_{\sim})=k=d_{\mathfrak{g}(G)}(u,v)$.
\end{proof}
\begin{proposition}
Let $G$ be a group.  Elements $\{g_1,g_2,\cdots,g_k\}$ form an independent set in $\mathfrak{g}(G)$ if and only if $\{[g_1]_{\sim},[g_2]_{\sim},\cdots,[g_k]_{\sim}\}$ is an independent set in $C(G)$.
\end{proposition}
\begin{proof}
Suppose $I=\{g_1,g_2,\cdots,g_k\}$ forms an independent set in $\mathfrak{g}(G)$.  Then for each $g_i,g_j\in I$, $\langle g_i\rangle \nleq \langle g_j\rangle$ and $\langle g_j\rangle \nleq \langle g_i\rangle$. So, $\{[g_i]_{\sim},[g_j]_{\sim}\}\notin E(C(G))$,  giving an independent set of size $k$ in $C(G)$.

Now suppose $\{[g_1]_{\sim},[g_2]_{\sim},\cdots,[g_k]_{\sim}\}$ is an independent set in $C(G)$. Then  $\{g_1,g_2,\cdots,g_k\}$ is an independent set in $\mathfrak{g}(G)$.
\end{proof}
\begin{Corollary}
$\alpha(\mathfrak{g}(G))=\alpha(C(G))$.
\end{Corollary}
\begin{proposition}
$\mathfrak{g}(G)$ is complete if and only if $C(G)$ is complete.
\end{proposition}
\begin{proof}
Suppose $\mathfrak{g}(G)$ is complete and let $[u]_{\sim}$ and $[v]_{\sim}$ be arbitrary vertices in $C(G)$ Since $u\in [u]_{\sim}$ and $v\in [v]_{\sim}$ and $\{u,v\}\in E(\mathfrak{g}(G))$, $\langle u\rangle \leq \langle v\rangle$ or $\langle v\rangle \leq \langle u\rangle$, so   $\{[u]_{\sim},[v]_{\sim}\}\in E(C(G))$.  Since $[u]_{\sim}$ and $[v]_{\sim}$ were arbitrary, $C(G)$ must be complete.

Now suppose $\mathfrak{g}(G)$ is not complete, then there exist vertices $u$ and $v$ such that $\{u,v\}\notin E(\mathfrak{g}(G))$.  Then $\{u,v\}$ is an independent set in $\mathfrak{g}(G)$ and consequently $\{[u]_{\sim},[v]_{\sim}\}$ is an independent set in $C(G)$.  Then there exist elements $[u]_{\sim}$ and $[v]_{\sim}$ in $C(G)$ such that $\{[u]_{\sim},[v]_{\sim}\}\notin E(C(G))$ so $C(G)$ is not complete.
\end{proof}
\begin{proposition}
Let $u_1,u_2,\cdots u_{2k}$ be vertices in $\mathfrak{g}(G)$, then $u_1, \cdots u_{2k}$ is a hole in $\mathfrak{g}(G)$ if and only if $[u_1]_{\sim},[u_2]_{\sim},\cdots,[u_{2k}]_{\sim}$ is a hole in $C(G)$.
\end{proposition}
\begin{proof}
Let $P=\{u_1,u_2,\cdots,u_{2k}\}$ be a hole in $\mathfrak{g}(G)$, then each $u_i\in P$ has exactly two neighbors in $P$, call them $u_{i+1}$ and $u_{i-1}$.  Then either $\langle u_i\rangle \leq \langle u_{i+1}\rangle$ and $\langle u_i\rangle \leq \langle u_{i-1}\rangle$ or $\langle u_{i+1}\rangle \leq \langle u_{i}\rangle$ and $\langle u_{i-1}\rangle \leq \langle u_{i}\rangle$.  Either way, the following are true, $\langle u_i\rangle \neq \langle u_{i+1}\rangle$ and $\langle u_i\rangle \neq \langle u_{i-1}\rangle$, so $[u_i]_{\sim}\neq [u_{i+1}]_{\sim}$ and $[u_i]_{\sim}\neq [u_{i-1}]_{\sim}$,     so $[u_{i+1}]_{\sim}$ and $[u_{i-1}]_{\sim}$ are neighbors of $[u_i]_{\sim}$ in $C(G)$. There is no chord in $C(G)$ inside that hole in $\mathfrak{g}(G)$, for if so, then there will be a chord inside the hole in $\mathfrak{g}(G)$ giving a contradiction. So $[u_1]_{\sim},[u_2]_{\sim},\cdots,[u_{2k}]$ forms a hole in $C(G)$.The other direction also follows via a similar argument.
\end{proof}
\begin{Corollary}
$C(G)$ are perfect graphs.
\end{Corollary}
\begin{proposition}
$\mathfrak{g}(G)$ is claw free if and only if $C(G)$ is so.
	\end{proposition}
	\begin{proposition}
		$\mathfrak{g}(Z_n)$ is chordal if and only if $C(Z_n)$ is chordal.
		\begin{proof}
			The result follows since a hole exists in $C(Z_n)$ if and only if a hole exists in $\mathfrak{g}(Z_n).$
		\end{proof}
		\end{proposition}
	\begin{proposition}
		A vertex $x\in \mathfrak{g}(Z_n)$ is simplicial if and only if  $[x]_{\sim}\in C(Z_n)$ is simplicial. 
		\begin{proof}
			The result follows as induced subgraph of a complete graph is complete.
		\end{proof}
	\end{proposition}
\begin{proposition}
If $C(G)$ is Hamiltonian, then $\mathfrak{g}(G)$ is also Hamiltonian.
\end{proposition}

\begin{proof}
Suppose $C(G)$ is Hamiltonian, and let $\{[u_1]_{\sim},[u_2]_{\sim},\cdots,[u_2]_{\sim}$, $[u_1]_{\sim}$ be a Hamiltonian cycle in $C(G)$.  $[u_i]_{\sim}$ is a clique in $\mathfrak{g}(G)$ since for $g_1,g_2\in [u_i]_{\sim}$, $\langle g_1\rangle =\langle g_2\rangle$, so $\langle g_1\rangle \leq \langle g_2\rangle$, so $\{g_1,g_2\}\in E(\mathfrak{g}(G))$, so denote $[u_i]_{\sim}$ by $\{u_{1i},u_{2i},\cdots,u_{k_ii}\}$, then $\{u_{11},u_{21},u_{31},\cdots,u_{k_11},u_{12},u_{22},$\\
$\cdots, u_{k_22},\cdots,u_{1n},u_{2n},\cdots u_{k_nn},u_{11}\}$ is a Hamiltonian cycle in $\mathfrak{g}(G).$
\end{proof}
Singh and Devi showed in \cite{CyclicSubgroupGraph} that the cyclic subgroup graph of cyclic groups of non-prime order is Hamiltonian in.  Power graphs of groups of prime order are complete, and therefore Hamiltonian for all orders except for $2$, so we note the following corollary by using R\'{e}dei's theorem,~\cite{Redei}.
\begin{Corollary}
Let $G\cong$ $\mathbb{Z}_n$ be a cyclic group with $n\neq 2$, then $\mathfrak{g}(G)$ is Hamiltonian.  
\end{Corollary}
Since the cyclic subgroup subgraph is isomorphic to an induced subgraph of the power graph, if $\mathfrak{g}(G)$  is planner, then so is $C(G)$. For example, consider $G= \mathbb{Z}_9$. Here $\mathfrak{g}(G)$ being a complete graph with $9$ vertices it is not planner. $C(G)$ being a triangle is planner. 

\section{Chordless Cycles}

It has now been shown that power graphs contain no holes of odd-length. Here it will be shown that for arbitrary even integer $n$, there exists a finite group whose power graph contains a hole of length $n$.

\begin{proposition}
Let $n$ be an even integer, then for even $n>4$, the power graph of a cyclic group will contain a hole of length $n$, if the order of the group has $\frac{n}{2}$ distinct prime factors.  The power graph of the group will contain a hole of length $4$, if the order of the group has at least two prime factors of multiplicity two or more.
\end{proposition}

\begin{proof}
First consider the case that $n=4$.  Then in the group $\mathbb{Z}_{p^2q^2}$ there is a subgraph consisting of the vertices $p,pq^2,q,$ and $p^2q$.  This subgraph will be a hole of length $4$.

Now consider the case that $n\geq 6$, and take primes $p_1,p_2,\cdots p_{\frac{n}{2}}$. Then the group contains a hole of length $n$ namely subgraph consisting of vertices $$p_1-p_1p_2- p_2-p_2p_3-p_3\cdots  p_{\frac{n}{2}} \cdots p_{\frac{n}{2}}p_{1}-p_1$$   
Then this subgraph is a hole of length $n$ in $\mathfrak{g}(\mathbb{Z}_{p_1p_2...p_{\frac{n}{2}}})$.
\end{proof}
Note: Following the proof of the theorem, it is possible to create many holes of even length permuting the positions of the primes and allowing various exponent of them.

\begin{proposition}
If the power graph of a finite cyclic group $G$ contains a hole of length $n$, then $\lvert G \rvert$ has at least $\frac{n}{2} $ distinct prime factors.
\end{proposition}

\begin{proof}
Suppose $\mathfrak{g}(\mathbb{Z}_m)$ contains a hole of length $n$. This cycle in the corresponding directed power graph consists of $\frac{n}{2}$ vertices with only out-edges and $\frac{n}{2}$ vertices with only in-edges.  Each vertex with out-edges is non-adjacent to each other vertex with out-edges, so certainly if $x, y$ are group elements represented by vertices in the hole with out-edges, then $\lvert x\rvert $ does not divide $\lvert y \rvert$, and $\lvert y \rvert$ does not divide $\lvert x\rvert $.  Then the order of each of these $\frac{n}{2}$ vertices with out-edges has a prime factor which is not shared by the other $\frac{n}{2}$ vertices with out-edges.  By Lagrange's theorem, the order of each element must divide the order of the group, so the order of the group must contain at least $\frac{n}{2}$ prime factors.  
\end{proof}

It has been shown that for an arbitrary even integer $n$, a finite group can be found whose power graph contains a hole of length $n$.  The proof relied on the fact that there were $\frac{n}{2}$ primes dividing the order of the group so that the multiples of the primes were a subset of their multiples in $\mathbb{Z}$.  Then $\mathfrak{g}(\mathbb{Z})$ contains holes of any even length.  In fact since there are infinitely many prime numbers, there will be an infinite number of holes of any even length in $\mathfrak{g}(\mathbb{Z})$.

A necessary and sufficient condition for the existence of a hole of length $n$ in finite cyclic groups has been given above.  This is not a necessary condition for the existence of holes of length $n$ in general Abelian groups.  Consider the Abelian group $\mathbb{Z}_{12}\times \mathbb{Z}_{12}$.  This group has order $144=2^4\cdot 3^2$.   consider the subgraph of $\mathbb{Z}_{12}\times \mathbb{Z}_{12}$ consisting of the elements $(1,0),(2,0),(1,6),(3,6),(1,2),(2,4),(1,8),(3,0)$. These elements (in order around the cycle) form a hole of length $8$.

\begin{proposition}
Let $G$ be an Abelian group whose order is $p^k$ for some prime $p$ and positive integer $k$.  Then $\mathfrak{g}(G)$ can contain no hole.
\end{proposition}
\begin{proof}
The result follows as the power graph is complete in this case. 
\end{proof}

\begin{Theorem}
An element whose order is a power of a prime cannot be a vertex with out-edges in a hole of even length.  
\end{Theorem}
\begin{proof}
	Suppose, that is not the case. Then, $\vec{\mathfrak{g}}(G)$ contains a hole with an element $x$ of order $p^n$ as a vertex with out-edges in the hole. Let $y$ and $z$ be  the elements adjacent to $x$.  Then $y$ and $z$ are of orders $p^l$ and $p^m$ respectively, where $l,m<n$.   $l\neq m$. Without loss of generality, suppose $l<m$, then $\langle y \rangle \subset \langle z \rangle$ as each are proper subgroups of $\langle x \rangle$ which is a primary cyclic group. So there is a chord from $y$ to $z$ giving a contradiction.
\end{proof}

\begin{proposition}
Power graphs of groups can contain no anti-holes of length greater than $4$.
\end{proposition}
\begin{proof}
Let $G$ be a group and suppose that $\mathfrak{g}(G)$ contains an anti-hole of length $n$ greater than $4$. Claim: All vertices in the anti hole have in-degree zero or out-degree zero in the corresponding directed power graph. Proof of the claim:
First note that, if we consider $n=4$, then the claim follows clearly. Arbitrarily choose a vertex $d$ in the anti hole. There must be a vertex $s_1$ in the antihole adjacent to $d$. Without loss of generality, let the source of the corresponding edge be $s_1$ and the destination vertex $d$. Now choose an edge between a vertex $s_2$, which is non-adjacent to vertex $s_1$, and adjacent to $d$.  Such an edge must exist as we assume $n>4$, since no two vertices can share the same pair of non-adjacent vertices.  The direction of this edge must be from $s_2$ to $d$ since if it were from $d$ to $s_2$ then a directed path would exist between $s_1$ and $s_2$, which by the path lemma would make them adjacent. Repeat the process for a vertex non-adjacent to one of either $s_1$ or $s_2$ and adjacent to $d$ again. The stopping point of this process is when all $n-3$ vertices adjacent to vertex $d$ have been selected. Following this procedure we see that $d$ has out-degree zero. Hence the claim follows.
Now power graph of a group being perfect, here $n>4$ means $n\geq 6$, as it can't have any anti-hole of odd length. So, for any arbitrary vertex $d$ in the antihole, degree of $d=n-3\geq 3$. So, it is possible to choose two vertices $v_1$ and $v_2$ in the neighborhood of $d$ that are adjacent to each other. Without loss of generality, if we assume $d$ has in degree zero, then as there is at least one directed edge between $v_1$ and $v_2$, we get a contradiction to the above claim. Hence the result follows.
  \end{proof}

\section{Completeness}
Here an alternative proof of the well known result regarding completeness of power graphs of cyclic groups of prime-power order in terms of the strong path lemma is presented.
\begin{proposition}
The power graph of a cyclic group of order $p^n$ where $p$ is prime and $n$ is a non-negative integer is complete.
\end{proposition}
\begin{proof}
The proof is by induction on $n$. By the Strong path lemma it suffices to show that there exists a directed path through all vertices in the power graph. When $n=0$ the graph consists of one vertex and the result follows.

Suppose a directed path exists through all vertices in the power graphs of cyclic groups of order $p^k$ for $0\leq k < n$.   Let $G\cong \mathbb{Z}_{p^k}$. Let $x,y$ denote two arbitrary generators of $G$, then in $\vec{\mathfrak{g}}(G)$ there is an edge from $x$ to $y$ and there is an edge from $y$ to $x$, as both elements are members of $G$ and therefore generated by each other. Then, there is a directed path between all generators of $G$.  This path can be extended towards a generator of a subgroup of order $p^{k-1}$, which contains a directed path through all of it's vertices by the induction hypothesis.  As every subgroup of $G$ divides the order of $G$ by Lagrange's theorem, every subgroup of $G$ is properly contained inside the subgroup of order $p^{k-1}$, so there exists a directed path through the entire vertex set of $\mathfrak{g}(G)$.  

By the strong path lemma, there is a clique of size $p^n$ in the cyclic group of order $p^n$, so the graph is complete.
\end{proof}

\section{Chromatic Number of Power Graphs of Cyclic Groups}
\begin{proposition}
    Let $G$ be the cyclic group of order $n$ and let $\mathfrak{g}(G)$ be its power graph.  Let $H$ denote the set of non-generators of $G$. Let $\chi(\Gamma)$ denote the chromatic number of a graph $\Gamma$. Then $\chi(\mathfrak{g}(G))=\phi(n)+\chi(\mathfrak{H})$ where $\mathfrak{H}$ is the subgraph of $G$ consisting of vertices representing non-generators and the edges between them and $\phi$ is Euler's totient function.
    \begin{proof}
        Since elements in $G\setminus H$ generate $G$, for any $g\in G\setminus H$ and $x\in G$, $x\in\langle g \rangle$, so vertex $g$ is adjacent to all elements in $\mathfrak{g}(G)$.  Then no color used in a coloring of the portion of $\mathfrak{g}(G)$ consisting of elements of $G\setminus H$ can be used in a coloring of $\mathfrak{H}$.  Additionally that portion of the graph is internally complete as well so the subgraph of $\mathfrak{g}(G)$ consisting of elements of $G\setminus H$ is isomorphic to $K_{\phi(n)}$ and has chromatic number $\phi(n)$.  Since every color used in the coloring of $\mathfrak{H}$ is distinct from every color in the subgraph of $\mathfrak{g}(G)$ consisting of elements of $G\setminus H$, and $\mathfrak{H}$ has chromatic number $\chi(\mathfrak{H})$, $\chi(\mathfrak{g}(G))$ is at most $\phi(n)+\chi(\mathfrak{H})$.  Also since the colors used in the colorings of the two subgraphs of $\mathfrak{g}(G)$ are distinct, $\chi(\mathfrak{g}(G))$ cannot be less than $\phi(n)+\chi(\mathfrak{H})$, so $\chi(\mathfrak{g}(G))=\phi(n)+\chi(\mathfrak{H})$.
    \end{proof}
\end{proposition}
\begin{Corollary}
    Let $G\cong \mathbb{Z}_n$ then $\phi(n)<\chi(\mathfrak{g}(G))\leq n$
\end{Corollary}

The result above gives a lower bound for the chromatic number of a power graph of a cyclic group by identifying a clique of a known size in every cyclic group.  By the strong path lemma, the existence of a directed path of length $k$ implies that the vertices along the path also make up a clique of length $k$.  If $m$ is the length of the longest directed path in a group $G\cong \mathbb{Z}_n$, then the chromatic number of the cyclic group of order $n$ is at least as large as the length of that path.
\begin{proposition}
Let $G$ be a group and let $m$ be the length of a directed path in $\vec{\mathfrak{g}}(G)$.  Then $m\leq \chi(\mathfrak{g}(G)) \leq n$.
\end{proposition}

In a cyclic group $G$ of order $n$, a path $m$ of length longer than $\phi(n)$ can be constructed as follows.  Follow the path of length $\phi(n)$ through the generators of of $G$.  After visiting the last generator, follow the path to a generator of a proper subgroup of $G$, which must exist as a path exists from a generator of $G$ to every element of $G$. As any two generators of any group have edges from each vertex to the other,  each of the generators of this subgroup can be added to the path.  This process can than be continued for a subgroup of this subgroup and so on until the the identity element is added to the path.  In fact, the longest path through any cyclic subgroup will be of this form, a descending chain of generators of proper subgroups.

\begin{Theorem}
Let $\mathbb{Z} / n\mathbb{Z}$ be the cyclic group of order $n$, and let $S_k$ be the set of generators of $\mathbb{Z} / k\mathbb{Z}$.  The longest path through $\mathbb{Z} / n\mathbb{Z}$ will be of the form \\$s_{n1},s_{n2},...s_{n\phi(n)}, d\cdot s_{\frac{n}{d}1}, d\cdot s_{\frac{n}{d}2}, d\cdot s_{\frac{n}{d}\phi(\frac{n}{d})}, ...(dd_1...d_j)\cdot s_{\frac{n}{dd_1d_2...d_j}}\phi(\frac{n}{dd_1d_2...d_j})$ where $d, d_1, d_2...d_j$ are prime divisors of $n$ with $d\leq d_1 \leq d_2 \leq ... \leq d_j$. and $s_{ix}\in S_i$.
\end{Theorem}
\begin{proof}
Suppose a path longer than the one given exists. Such a path necessarily begins with the set of elements of order $n$, that is the elements of $S_n$, since if it did not these elements could simply be added to the beginning of the path to make a new longer path.  After every element of order $n$ is added to the path an element $g_1$ from a proper subgroup of $\mathbb{Z} / n\mathbb{Z}$ can be added to the path, but this element will have order $\frac{n}{d}$ where $d$ is a divisor of $n$.  In fact this element will generate a subgroup $\langle g_1 \rangle$ of order $\frac{n}{d}$ which is isomorphic to $\mathbb{Z} / \frac{n}{d}\mathbb{Z}$ with the mapping given by $z\in \mathbb{Z} / \frac{n}{d}\mathbb{Z} \mapsto d\cdot g\in \langle g_1 \rangle$.  Then for each of the $\phi(\frac{n}{d})$ generators $z\in z\in \mathbb{Z} / \frac{n}{d}\mathbb{Z}$, the corresponding element $d\cdot z \in \langle g_1 \rangle$must be added to the path, since if it were not added to the path a new path could be constructed with these elements following (or preceding) $g_1$ which is longer. The same process can be repeated from $\langle g_1 \rangle$, add an element from a subgroup of $\langle g_1 \rangle$ and all of the other generators of the same subgroup.  In this way generators of a descending chain of subgroups are added to the path, terminating with the generator of the trivial group, the identity element.  

Since the longest path must be in the form of generators of a chain of subgroups, it remains to be shown that by always choosing the largest possible proper subgroup whose generators to add to the path, the path size is maximized.  That is, by traversing the subgroups of $\mathbb{Z} / n\mathbb{Z}$ in the order \\$\mathbb{Z} / n\mathbb{Z}\to d\mathbb{Z} / n\mathbb{Z}\to d\cdot d_1\mathbb{Z} / n\mathbb{Z}\to...\to (d\cdot d_1\cdot ... \cdot d_j)\mathbb{Z} / n\mathbb{Z}$ where $d\leq d_1 \leq ... \leq d_j$, the number of elements added to the path is as large as possible.  To see this, some properties of Euler's totient function are examined.  First observe that increasing any single prime factor in a number will increase the totient of that number, that is if $n=p_{11}p_{12}...p_{1k}$ where $p_1 \leq p_2 \leq... \leq p_k$ and $m=p_{21}p_{22}...p_{2k}$ where $p_{1x}=p_{2x}$ for all $x$ except one, and at that one index $p_{2x}>p_{1x}$, then $\phi(m)\geq \phi(n)$.  It is known that the totient function is multiplicative over relatively prime arguments, that is $\phi(ab)=\phi(a)\phi(b)$ if $\gcd(a,b)=1$, so we can write $\phi(n)$ and $\phi(m)$ as $\Pi_{x\in X} \phi(p_{1x}^{k_x})$ and $\Pi_{y\in Y} \phi(p_{2y}^{k_y})$ respectively where $X$ is the set of prime factors of $n$ and $Y$ is the set of prime factors of $m$.  Then both $\phi(n)$ and $\phi(m)$ can be divided by the prime factors and multiplicities for which they agree, leaving only the prime factors which differ between $n$ and $m$.  Then observe that $\phi(p^k) < \phi(q^k)$ if $p < q$ are distinct primes.  Also $\phi(p^{k})\leq \phi(p^{k-1})q$ if $p< q$ since $p^k-p^{k-1} \leq (p^{k-1}-p^{k-2})(q-1)$.  Then it is clear that if $n$ is the order of a cyclic group the subgroups must be traversed in the order above to make the number of elements in the path as large as possible.
\end{proof}
Here the largest path through $\vec{\mathfrak{g}}(\mathbb{Z} / n\mathbb{Z})$ has been constructed, which will have a length equal to the size of the largest clique in $\mathfrak{g}(\mathbb{Z} / n\mathbb{Z})$, by the path-clique equivalence theorem.  Since power graphs are perfect, the size of the largest clique in a power graph is also equal to its chromatic number.  Then the following result is true for power graphs of cyclic groups

\begin{Corollary}
Let $G\cong \mathbb{Z} / n\mathbb{Z}$ be a group with power graph $\mathfrak{g}(G)$.  Also let $d_1\leq d_2\leq... \leq d_k$ be (not necessarily distinct) prime divisors of $n$ then $\chi(\mathfrak{g}(G))=\phi(n)+\phi(\frac{n}{d_1})+\phi(\frac{n}{d_1d_2})+...+\phi(\frac{n}{d_1d_2...d_k})$
\end{Corollary}

Here the proof of Theorem \ref{theo 3} did not depend on the fact that $G$ was a cyclic group in any way, except that the path constructed through the elements of $G$ could always be started with an element of order $n$, where $n$ is the order of $G$.  In a general group, there are no elements with order equal to the order of the group, so there are many possibilities for where to begin the longest path.  Let $g_1, g_2,...g_n$ be elements of a group $G$, and define $\Psi(n)=\phi(n)+\phi(\frac{n}{d_1})+\phi(\frac{n}{d_1d_2})+...+\phi(\frac{n}{d_1d_2...d_k})$, where $d_1, d_2, ...d_k$ are prime divisors of $n$, then $\chi(G)=\max_{i}\{\Psi(g_i)\}$

\section{Chordallity of Power Graph of cyclic groups}

		\begin{proposition}
			Consider $n\neq p^{m}$ for some prime $p$ and a positive integer $m$. If a vertex $k \in \mathfrak g(\mathbb Z_{n})$ is simplicial then $gcd(k,n)\neq 1$.
			\begin{proof}
				Let $k \in \mathfrak g(\mathbb Z_{n})$ be simplicial with $gcd(k,n)= 1$. Then $k$ generates $\mathbb Z_{n}$. So, $k$ being adjacent to every $x\in \mathfrak g(\mathbb Z_{n})$, $\mathfrak g(\mathbb Z_{n})$ turns to be complete, which is a contradiction as $n\neq p^m$.
			\end{proof}
		\end{proposition}
		
		Converse of the above result is not true in general. For example: Consider $Z_{12}, gcd(6,12)\neq 1$. Though, $6$ is not a simplicial vertex in $\mathfrak g(\mathbb Z_{12})$ as because $2,3$ are adjacent to $6$, but they are not adjacent to each other.

		\begin{proposition}\label{prop 27}
			$\mathfrak g(\mathbb Z_{n})$ is chordal if and only if $n=p^{m}$ for some prime $p$ and positive integer $m$ or $n=p^{m}q$, for two distinct primes $p,q$ and positive integer $m$.
		\end{proposition}
		\begin{proof}
			If $n=p^{m}$, then $\mathfrak g(\mathbb Z_{n})$ being complete is chordal. If $n=p^{m}q$, then subgroup diagram of the group $\mathbb Z_{n}$ is chordal. And hence, $C(Z_n)$ is also so, as the subgroup diagram is a subgraph of   $C(Z_n)$. Thus, $\mathfrak g(\mathbb Z_{n})$ is chordal.
			Conversely, if $n\neq p^{m}, p^{m}q$, then $n$ has at least two distinct prime factors $p,q$ with powers $m,n$ where both $m,n$ are positive integers bigger than or equal to $2$. In that case, $p-pq^{2}-q-p^{2}q-p$ is a chordless cycle in $\mathfrak g(\mathbb Z_{n})$  giving the graph as  non-chordal.
		\end{proof}

		\begin{proposition}
			A vertex in $C(Z_{n})$ other than $[n]_\sim$ and $[1]_\sim$ is simplicial iff it has only one parent and one child.
			
			\begin{proof}
				Let $x$ be a vertex in $C(Z_{n})$ other than $[n]_\sim$ and $[1]_\sim$. If it has more than one parent, then any two parents are not adjacent. Similarly, if it has more than one child then any two children are not adjacent to each other.\\
				Conversely let $x$ be a simplicial vertex. Then, it can neither have more one parent, nor more than one child.
			\end{proof}
		\end{proposition}
		Note that, even if $\mathfrak g(\mathbb Z_{n})$ is not chordal for $n= p^{m}q^{r}$ where $m,r \geq 2$ and $p,q$ are distinct primes, they have simplicial vertices namely $p^{m}, q^{r}$, as because they are simplicial in the corresponding $C(Z_n)$.
		
		Thus, we can now state the following result:
		\begin{Theorem}
		If $n = {\prod_{i=1}^{k} p_{i}^{\alpha_{i}}}$,  then whenever $k\geq 3, \mathfrak g(\mathbb Z_{n})$ does not have any simplicial vertex.
		\end{Theorem}
		\begin{proof}
			$n$ and $1$ are not simplicial. Other wise, $\mathfrak g(\mathbb Z_{n})$ will be complete, giving more than one prime factor of $n$ which contradicts the hypothesis of the theorem. On the other hand, if we consider any vertex in $\mathfrak g(\mathbb Z_{n})$ namely $ m={\prod_{i=1}^{k} p_{i}^{\beta_{i}}}$,
			
			  then the equivalence classes of  $p_{j}p_{i},p_{j}p_{k}$ are two parents of that of  $p_{j}$ in $C(\mathbb Z_{n})$. So, the vertex is not simplicial in $C(\mathbb Z_{n})$ and hence is not simplicial in $\mathfrak g(\mathbb Z_{n})$.\\
			 \end{proof}

\section{Power graph of $U_n$ and $Q_n$}
For any positive integer $n$, let $\mathbb U_{n}$ be the multiplicative group of integers modulo $n$ and let $ \mathbb Q_n$ be it's subgroup of quadratic residue modulo $n$. Then, we have the following results.

\begin{proposition}
	$\mathfrak g(\mathbb Q_{n})$ is not planner whenever,

\begin{enumerate}
		\item [i.] $n=p$ or $2p$, where $p$ is a prime bigger than $37$.
		
		\item[ii.] $n=p^{m}$ or $2p^{m}$, where $p\geq 7, m\geq 2$ or else, $p=3$ or $5$ and $m\geq 2$.
		
		\end{enumerate}
\end{proposition}		
\begin{proof}$\phantom e$		
\begin{enumerate}
		\item [i.]	 For $n=p$ or $2p$, the cardinality of $\mathbb Q_{n}$ is $\mu =\frac {\phi(n)}{2}= \frac{p-1}{2}$ by Lemma 7.3 of \cite{book}. In that case $\mathbb Q_{n}$ is cyclic as $\mathbb U_{n}$ is so. Hence by lemma 4.7 in \cite{power semi}, $\mathfrak g(\mathbb Q_{n})$ is not planner if $\phi{(\frac{p-1}{2})}>7$, that is if $\frac{p-1}{2}>18$, that is if $p>37$.
			
\item[ii.] Now let $n=p^{m}$, where $p$ is an odd prime, then by lemma 4.7 in \cite{power semi}, $\mathfrak g(\mathbb Q_{n})$ is not planner if $\frac {\phi(n)}{2}= \frac{p^{m}}{2}= \frac {p^{m-1}(p-1)}{2}= p^{m-2}(p-1){\phi(\frac{p-1}{2})}>7$. Hence, the result follows. 
\end{enumerate}			
\end{proof}
\begin{proposition}
	$\mathfrak g(\mathbb Q_{n})$ is planner if $n$ divides $240$. 
\end{proposition}
\begin{proof} If $n$ divides $240$, then $\mathfrak g(\mathbb U_{n})$ is planner by theorem 4.10 in \cite{power semi}. Hence by proposition 4.5 in \cite {power semi}, $\mathfrak g(\mathbb Q_{n})$ is also planner. 
\end{proof}
\begin{proposition}
	Let $n= p^{m}$ or $n=2p^{m}$, where $p$ is a Fermat prime. Then, $\mathfrak g(\mathbb U_{n})$ is chordal iff $m\leq 2$ or $p=F_{0}=3$.
	\end{proposition}
	\begin{proof}
		Let $n= p^2$ or $n=2p^2$ where $p$ is a Fermat prime. Then, by Corollary 6.14 in \cite{book}, $\mathbb U_{n}$ is isomorphic to $\mathbb Z_{p}\times \mathbb Z_{p-1}= \mathbb Z_{p}\times \mathbb Z_{2^{m}}=\mathbb Z_{2^{m}p}$, for some positive integer $m$ as $p$ is a Fermat Prime. Hence the graph is chordal by proposition  \ref{prop 27}. In a similar way, for $n=p$ or $n=2p$ where $p$ is a Fermat prime, $\mathbb U_{n}$ is isomorphic to $\mathbb Z_{p-1}=\mathbb Z_{2^{m}}$ for some positive integer $m$ and hence is chordal by proposition \ref{prop 27}. Now let $p=F_{0}=3$. Then, if $n=p^{m}$ or $n=2p^{m}, U_{n}$ is isomorphic to $\mathbb Z_{p^{m-1}}\times\mathbb Z_{2}=\mathbb Z_{2p^{m-1}}$, as $p$ is odd. Hence the graph is chordal.\\
		Conversely, let $n\neq p,2p,p^{2},2p^{2}$ where $p\neq 3$ and let $n\neq p^{m}, 2p^{m}$ where $p=F_{0}$ and $m$ is any positive integer. So, here if $p>F_{0}, m\geq 3$, then as in \cite{book}, $\mathbb U_{n}$ is isomorphic to $Z_{p^{m-1}}\times\mathbb Z_{2^{i}}, i\geq 2$. Thus the power graph is not chordal by proposition~\ref{prop 27}.
			\end{proof}
			\begin{Corollary}
				For a positive integer $n$, if $p^{m}$ or $2p^{m}$ divides $n$, where $p$ is a Fermat prime bigger than $3$ and $m\geq 3$, then $\mathfrak g(\mathbb U_{n})$ is not chordal. 
			\end{Corollary}
			\begin{proposition}\label{prop 32}
				Let, $n$ be an odd integer which is not square free and no Fermat prime be a factor of $n$. Then, $\mathfrak g(\mathbb U_{n})$ is not Chordal.
			\end{proposition}
			\begin{proof}
				As $n$ is not square free, there is an odd prime $p$, which is not a Fermat prime and $p^{m}$ divides $n$ for some $m\geq 2$. In that case, as in Corollary 6.14 of \cite{book}, $\mathbb Z_{p^{m-1}(p-1)}$ appears inside the direct product decomposition of $\mathbb U_{n}$, where $p-1$ is not a power of $2$ since $p$ is not a Fermat prime. So, $p-1$ being it has an odd prime factor $q$ where $p\neq q$, thus $p^{m-1}(p-1)$ contains at least $3$ distinct odd primes and hence the power graph is not chordal.  
			\end{proof}
			\begin{proposition}\label{prop 33}
				Let $n$ be an even integer. Then, if 
\begin{enumerate}
\item [i.] $2^{f}\parallel n$ (that means $f$ is the largest integer so that $2^{f}$ divides $n$), where $f\geq 4$ and if there are at least two distinct primes $p$ and $q$ and positive integers $\alpha, \beta$, where neither of $p,q$ are Fermat's prime, then $\mathfrak g(\mathbb U_{n})$ is not chordal.
\item[ii.] If $2^{f}\parallel n, f\geq 1$ and $p^{\alpha}, q^{\beta} \parallel n$, where $\alpha,\beta \geq 2$ and neither of $p,q$ are Fermat's prime, then
					 $\mathfrak g(\mathbb U_{n})$ is not chordal.
				\end{enumerate} 
			\end{proposition}
\begin{proof}
	The proof follows by Corollary 6.14 in \cite{book} and by our proposition \ref{prop 27}.
\end{proof}
We have analogous result regarding chordality for $\mathfrak g(\mathbb Q_{n})$.

\begin{proposition}
Let $n= p^{m}$ or $n=2p^{m}$, where $p$ is a Fermat prime. Then, $\mathfrak g(\mathbb Q_{n})$ is chordal iff $m\leq 2$ or $p=F_{0}=3$ or $p=F_{1}=5$.
\end{proposition}
\begin{proof}
	The proof follows by using Lemma 7.1 and Corollary 6.14 in \cite{book} as in that case, $\mathbb Q_{n} =\mathbb Z_{\frac{\phi(p^{m})}{2}}= Z_{p^{m-1}(\frac{p-1}{2})}$. Hence, by similar method as in the proof of proposition \ref{prop 32}, the result follows. 
\end{proof}
\begin{Corollary}
	For a positive integer $n$, if $p^{m}$ or $2p^{m}$ divides $n$, where $p$ is a Fermat prime bigger than $3, 5$ and $m\geq 3$, then $\mathfrak g(\mathbb Q_{n})$ is not chordal. 
\end{Corollary}
\begin{proposition}
	Let, $n$ be an odd integer which is not square free and no Fermat prime be a factor of $n$. Then, $\mathfrak g(\mathbb Q_{n})$ is not Chordal.
\end{proposition}
\begin{proof}
	The proof follows as in proposition \ref{prop 32} and Lemma 7.1 in \cite{book}.
\end{proof}

	\begin{proposition}
		Let $n$ be an even integer. Then, if \begin{enumerate}
			\item[i.] $2^{f}\parallel n$  where $f\geq 5$ and if there are at least two distinct primes $p$ and $q$ and positive integers $\alpha, \beta$, where neither of $p,q$ are Fermat's prime, then $\mathfrak g(\mathbb Q_{n})$ is not chordal.
			\item[ii.] If $2^{f}\parallel n, f\geq 2$ and $p^{\alpha}, q^{\beta} \parallel n$, where $\alpha,\beta \geq 2$ and neither of $p,q$ are Fermat's prime, then
			$\mathfrak g(\mathbb Q_{n})$ is not chordal.
		\end{enumerate} 
	
	\end{proposition}
	\begin{proof}
		The proof follows as in proposition \ref{prop 33}.
	\end{proof}
	Finally, we conclude with a question: Precisely, for what positive integers $n,\mathfrak g(\mathbb U_{n})$ and $\mathfrak g(\mathbb Q_{n})$ are chordal?
	
\textit{Acknowledgment:\/} Authors acknowledge Dr. M. K. Sen for his helpful suggestions and advice towards this work.

\end{document}